\documentclass[twoside,a4paper,reqno,11pt]{amsart}

\usepackage{amsfonts,amsmath,amssymb,amsthm}
\usepackage{amscd}
\usepackage{fullpage,enumitem,color}
\usepackage{amscd,multicol}
\usepackage[backend=bibtex, style=trad-alpha]{biblatex}

\setlist[enumerate]{itemsep=3pt,topsep=3pt}
\setlist[enumerate,1]{label=\textup{(\roman*)}}
\setlist[enumerate,2]{label=\textup{(\alph*)}}

\newtheorem{lemma}{Lemma}[section]
\newtheorem{theorem}[lemma]{Theorem}

\newtheorem{corollary}[lemma]{Corollary}

\newtheorem{maintheorem}{Theorem}

\setlist[enumerate]{topsep=0pt}
\bibliography{JRT}

\newcommand{\C}{\mathbb{C}}

\renewcommand{\O}{\mathbb{O}}
\renewcommand{\P}{\mathbb{P}}
\newcommand{\Q}{\mathbb{Q}}
\newcommand{\R}{\mathbb{R}}

\newcommand{\Z}{\mathbb{Z}}

\newcommand{\CP}{\C\P}

\newcommand{\eps}{\epsilon}
\renewcommand{\o}{\otimes}

\newcommand{\Gr}{\mathrm{Gr}}

\newcommand{\Out}{\mathrm{Out}}

\newcommand{\Spin}{\mathrm{Spin}}

\newcommand{\SO}{\mathrm{SO}}

\newcommand{\Th}{\mathrm{Th}}

\newcommand{\from}{\leftarrow}

\title{On the $32$-dimensional Rosenfeld projective plane}

\author{John Jones, Dmitriy Rumynin, Adam Thomas}

\begin{document}

\begin{abstract}
Following on from \cite{JRTI}, we make a detailed study of the $32$-dimensional Rosenfeld projective plane which is the symmetric space EIII in Cartan's list of compact symmetric spaces. 
\end{abstract}

\maketitle

\section*{Introduction} 
In \cite{JRTI} we made a systematic study of the classical topological invariants of homogeneous spaces with a particular emphasis on the twelve compact symmetric spaces for the exceptional Lie groups.  Our main examples in that paper are the three Rosenfeld projective planes of dimensions $32, 64$ and $128$.  In this paper we make a much more detailed study of the $32$-dimensional Rosenfeld projective plane.  %For the general context of this work see \cite{JRTI}.

Throughout we will use the notation $R$ for the $32$-dimensional Rosenfeld projective plane. Explicitly,
$$
R = \frac{E_6}{\Spin(10)\times_{C_4} S^1},
$$
which is a symmetric space for the exceptional Lie group $E_6$.  It is called $R_5$ in \cite{JRTI}, EIII in Cartan's list of compact symmetric spaces, and also known as $P^2(\O \otimes \C)$. The subgroup $C_4$ is the cyclic subgroup of $\Spin(10) \times S^1$ generated by $(\eps,i) \in \Spin(10) \times S^1$. Here $\eps$ is the element of the centre of $\Spin(10)$ which acts as multiplication by $i$ on $\delta_{10}^+$ and, therefore, multiplication by $-i$ on $\delta_{10}^-$.  As usual, $\delta_{10}^{\pm}$ are the two $16$-dimensional complex spin representations of $\Spin(10)$. 

We also use the following notation 
$$
P = \frac{F_4}{\Spin(9)}, \quad S = \frac{F_4}{\Spin(7)}, \quad Q = \frac{F_4}{\Spin(7) \times_{C_2} \Spin(2)}.
$$
These three homogeneous spaces for $F_4$ have dimensions  $16, 31, 30$ respectively, and all three are submanifolds of $R$.  

We study $R$ by using the action of $F_4 \subset E_6$ on $R$. The idea for this approach comes from Atiyah and Berndt \cite{MR2039984}.  The action of $F_4$ on $R$ has three orbit types. The principal (generic) orbit of codimension $1$ is $S$ and there are $2$ special orbits $P$ and $Q$.  This orbit structure  is described in the Appendix to \cite{MR2039984} and the paper \cite{LANDSBERG2001477}. It can also be derived from \cite[Chapter~14]{MR1428422}. In standard terminology, the action of $F_4$ on $R$ has cohomogeneity $1$, and the orbit space $R/F_4$ is a closed interval.

Using Mostert's cohomogeneity one theorem \cite{MR0085460,MR0095897,MR3778980,bredon1972introduction} we get the following result. 

\begin{maintheorem} \label{mainthm:Mostert} Let $N_P$, $N_Q$ be the normal bundles of $P$ and $Q$ in $R$.  Let $D_P$, $D_Q$ be the disc bundles in $N_P$, $N_Q$ and $S_P$, $S_Q$ be the corresponding sphere bundles.  There are diffeomorphisms 
$$
\begin{CD}
S_P @<{e_1}<< S  @>{e_2}>> S_Q
\end{CD}
$$
such that
$$
R = D_P \cup_{S} D_Q.
$$
\end{maintheorem}

An equivalent way to describe $R$ is as the double mapping cylinder of the maps $P \from S \to Q$ given by the projections in the sphere bundles $S_P$ and $S_Q$.%, see \cite[Chapter~IV Theorem~8.2]{bredon1972introduction}  

The $16$-dimensional real vector bundle $N_P$ is the bundle over $P = F_4/\Spin(9)$ associated to the $16$-dimensional real spin representation of $\Spin(9)$.  This is also isomorphic to the tangent bundle of $P$.  The $2$-dimensional real bundle $N_Q$ over $Q = F_4/(\Spin(7) \times_{C_2} \Spin(2))$ is associated to the obvious representation $\Spin(7)\times_{C_2} \Spin(2) \to \SO(2)$. 

This geometric decomposition of $R$ allows us to give a new proof of a theorem originally proved by Toda and Watanabe \cite{MR0358847}.

\begin{maintheorem} \label{mainthm:intcoh}
The integral cohomology ring of $R$ is 
\[
H^*(R) = \frac{\Z[t,w]}{(r_{18}, r_{24})}
\]
where $t \in H^2(R)$, $w \in H^8(R)$ and the relations are
$$
r_{18} = t^9 - 3w^2t, \quad r_{24} = w^3 - 9wt^8 + 15 w^2t^4.
$$
\end{maintheorem}

There is an equivalent way to present the cohomology of $R$ using Poincar\'e duality.  Choosing a generator (fundamental class) $[R]$ in $H_{32}(R) \cong \Z$ yields a non-degenerate bilinear pairing
$$
\mu : H^p(R) \o H^{32-p}(R) \to \Z, \qquad \mu(a,b) = \langle ab, [R] \rangle.
$$
This tells us that if we know all products that end up in the top degree, then we know all products. The following table, proved in Theorem \ref{thm:prodsR}, gives all products to the top degree.

\begin{center}
\begin{tabular}{ | c | c | c | c | c | }
\hline
$t^{16}$ & $t^{12}w$ & $t^8w^2$ & $t^4w^3$ & $w^4$\\
\hline
$3 \cdot 26$ & $3 \cdot 15$ & $26$ & $15$ & $9$ \\
\hline
\end{tabular}
\end{center}

Recall that $R$ is a $16$-dimensional smooth complex subvariety of $\CP^{26}$, often called the fourth Severi variety \cite{MR0773432}. It is also a generalised flag variety, see \cite[\S1b]{Lazar}.  We write $T_c(R)$  for the 16-dimensional complex tangent bundle of $R$.  As an application, we explain how to calculate the Chern classes of $T_c(R)$. 

\begin{maintheorem} \label{mainthm:chernclasses}
The Chern classes of $T_c(R)$ are as follows. 
\end{maintheorem}
\vspace{-0.9cm}
\begin{center}
  \begin{minipage}[t]{.45\textwidth}
\begin{align*}  
c_1 & = 12 t, \\
c_2 & = 69 t^2, \\
c_3 & = 252 t^3, \\
c_4 & = 657 t^4 - 6 w, \\
c_5 & =    1296 t^5 - 36 t w, \\
c_6 & =    1995 t^6 - 102 t^2 w, \\
c_7 & =    2448 t^7 - 198 t^3 w, \\
c_8 & =    2412 t^8 - 288 t^4 w + 39 w^2,
\end{align*}
  \end{minipage}
  \quad
  \begin{minipage}[t]{.45\textwidth}   
\begin{align*}
c_9 & =  -270 t^5 w + 5760 t w^2, \\
c_{10} & =  -180 t^6 w + 3645 t^2 w^2, \\
c_{11} & =    -432 t^7 w + 2430 t^3 w^2, \\
c_{12} & =    750 t^4 w^2 - 136 w^3, \\
c_{13} & =    360 t w^3, \\
c_{14} & =    84 t^2 w^3, \\ 
c_{15} & = 1512 t^3 w^3 - 864 t^7 w^2 \\
c_{16} & =    3 w^4.
\end{align*}
  \end{minipage}
\end{center}

In Section \ref{sec:triality} we explain the role that the triality automorphism of $\Spin(8)$ plays in this study of $R$. Section~\ref{sec:HQ} is devoted to the calculation of $H^*(Q)$. In Section~\ref{sec:HR} we give a proof of Theorem~\ref{mainthm:intcoh}. Finally, in Section~\ref{sec:chern} we turn to the calculation of the characteristic classes of vector bundles over $R$, proving Theorem~\ref{mainthm:chernclasses}.

\subsection*{Acknowledgments}
This research was funded in part by the EPSRC, EP/W000466/1 (Thomas). For the purpose of open access, the author has applied a Creative Commons Attribution (CC BY) licence to any Author Accepted Manuscript version arising from this submission.

%In Section 2 we gather together some straightforward preliminaries we will need.

\section{Triality and homogeneous spaces of $F_4$} \label{sec:triality}
One of the special features of $\Spin(8)$ is triality. Recall that $\Spin(8)$ has three $8$-dimensional real representations, the vector representation $u_8$ and the two spin representations $\delta^{\pm}_8$.  No two of these representations are isomorphic but  given any two there is an outer automorphism of $\Spin(8)$ which transforms one to the other. Indeed the group $\Out(\Spin(8))$ of outer automorphisms of $\Spin(8)$ can be identified with $\Sigma_3$, the group of permutations of the set $\{u_8, \delta_8^+, \delta_8^-\}$. 

The representation $u_8$ gives a transitive action of $\Spin(8)$ on $S^7$ with stabiliser $\Spin(7)$. By triality, the same is true for $\delta_8^{+}$ and $\delta^-_8$.  So we get three (conjugacy classes) of embeddings 
\[
i, j_+, j_- : \Spin(7) \to \Spin(8),
\]
and each of the homogeneous spaces 
\[ \frac{\Spin(8)}{i(\Spin(7))}, \quad \frac{\Spin(8)}{j_+(\Spin(7))}, \quad \frac{\Spin(8)}{j_-(\Spin(7))} \] 
are $\Spin(8)$-equivariariantly diffeomorphic to $S^7$.  

Now we have the usual embeddings
$$
\Spin(8) \to \Spin(9) \to F_4.
$$
The embeddings $j_+$, $j_-$ are conjugate in $\Spin(9)$ and we get two conjugacy classes of embeddings
$$
i, j : \Spin(7) \to \Spin(9).
$$
The embeddings $i, j$ are conjugate in $F_4$, see \cite{AdamsTrialityallthat}.

First we identify the homogeneous spaces
$$
\frac{\Spin(9)}{i(\Spin(7))}, \quad \frac{\Spin(9)}{j(\Spin(7))}.
$$

\begin{lemma}
\begin{enumerate}
\item
The homogeneous space $\Spin(9)/i(\Spin(7))$ is $\Spin(9)$-equivariantly diffeomorphic to the Stiefel manifold $V_2(\R^9)$ where $\Spin(9)$ acts on $V_2(\R^9)$ via the vector representation.
\item
The homogeneous space $\Spin(9)/j(\Spin(7))$ is $\Spin(9)$-equivariantly diffeomorphic to $S^{15}$ where $\Spin(9)$ acts on $S^{15}$ as the sphere in the $16$-dimensional real spin representation of $\Spin(9)$.
\end{enumerate}
\end{lemma}

\begin{proof}
The embedding $i$ is conjugate to the usual embedding of $\Spin(7)$ in $\Spin(9)$, and so the homogeneous space $\Spin(9)/i(\Spin(7))$ is the Stiefel manifold $V_2(\R^9)$ of $2$-frames in $\R^9$.  The spin representation of $\Spin(9)$ is a $16$-dimensional real representation and $\Spin(9)$ acts transitively on $S^{15}$, the sphere in the spin representation.  This action is transitive and the stabiliser of a point is (conjugate to)  $j(\Spin(7))$, see \cite{Bryantspinor}.
\end{proof}

Next we identify the homogeneous spaces 
$$
\frac{F_4}{i(\Spin(7))}, \quad \frac{F_4}{j(\Spin(7))}.
$$
The first is the total space of the fibre bundle
$$
\frac{\Spin(9)}{i(\Spin(7))} \to \frac{F_4}{i(\Spin(7))} \to \frac{F_4}{\Spin(9)}.
$$
The fibre of this bundle is the Stiefel manifold $V_2(\R^9)$.  Let $U_9$ be the real $9$-dimensional real vector bundle over $F_4/\Spin(9)$ associated to the $9$-dimensional vector representation of $\Spin(9)$. It follows from the previous lemma that $F_4/i(\Spin(7))$ is $F_4$-equivariantly diffeomorphic to the fibrewise Stiefel manifold $V_2(U_9)$, that is the fibre bundle over $P$ with fibre over $x \in P$ equal to $V_2(U_{9, x})$, where $U_{9, x}$ is the fibre of $U_9$ over $x$.

The second is the total space of the fibre bundle
$$
\frac{\Spin(9)}{j(\Spin(7))} .\to \frac{F_4}{j(\Spin(7))} \to \frac{F_4}{\Spin(9)}.
$$
Let $\Delta_{9}$ be the $16$-dimensional real vector bundle over $P = F_4/\Spin(9)$ associated to the (real) spin representation of $\Spin(9)$. This time the previous lemma tells us that $F_4/j(\Spin(7))$ is $F_4$-equivariantly diffeomorphic to $S(\Delta_{9})$, the sphere bundle of $\Delta_{9}$.  

As mentioned above, the embeddings $i, j : \Spin(7) \to F_4$ are conjugate. This proves the following theorem.

\begin{theorem}
There are $F_4$-equivariant diffeomorphisms
\[
S(\Delta_{9}) \xleftarrow{\hphantom{he} e_1 \hphantom{he}   } \frac{F_4}{\Spin(7)} \xrightarrow{\hphantom{he} e_2 \hphantom{he}   }  V_2(U_9)
\]
\end{theorem}

This theorem is quite surprising. We get two fibre bundles 
$$
S^{15} \to S(\Delta_{9}) \to P, \quad V_2(\R^9) \to V_2(U_9) \to P
$$ 
with diffeomorphic total spaces and bases. Clearly, since the fibres of these bundles are not diffeomorphic, there is no fibre preserving diffeomorphism. The $E_2$ pages of the Serre spectral sequences of the two fibre bundles look very different but they converge to the same answer.

\section{The integral cohomology groups of $Q$.} \label{sec:HQ}

\subsection{The integral cohomology of $P$ and $S$}
The homogeneous space $P = F_4/\Spin(9)$ is the Cayley projective plane.  Its integral cohomology and Pontryagin classes are well known, see  \cite[\S 19]{MR0102800}.  As a ring 
$$
H^*(P) = \frac{\Z[a]}{(a^3)} \quad \text{ with } a \in H^8(P).
$$
Furthermore $a$ can be chosen so that the total Pontryagin class $p(TP)$ and the Euler class $e(TP)$ are given by
$$
p(TP) = 1 + 6a +39a^2 \quad \text{ and } \quad e(TP) = 3a^2.
$$

The homogeneous space $S$ is the sphere bundle $S(TP)$. A simple argument with the Gysin sequence of this sphere bundle shows that
$$
H^i(S) =
\begin{cases} 
 \Z \quad & \text{if $i= 0, 8, 23, 31$,} \\
\Z/3         &\text{if $i = 16$,}\\
0             &\text{ otherwise}.
\end{cases}
$$

\subsection{The integral cohomology groups of $Q$} \label{sec:intcohQ}
An argument along the lines of \cite[Section 7.4]{JRTI}, shows that 
$$
H^*(Q; \Q) = \frac{\Q[a_2, a_8]}{(\rho_{16}, \rho_{24})}
$$
and its Poincar\'e polynomial is 
$$
(1 + x^2 + x^4 + \dots + x^{14})(1 + x^8 + x^{16}).
$$
We also know from \cite[Theorem~A]{MR87035} that $H^*(Q)$ is torsion free.

Let $q : Q \to P$ be the fibre bundle with fibres diffeomorphic to the Grassmannian of oriented $2$-planes in $\R^9$ \[ \Gr_2(\R^9) = \frac{\Spin(9)}{\Spin(7)\times_{C_2}\Spin(2)},\] also known as the complex quadric. 

There are elements $e \in H^2(\Gr_2(\R^9))$ and $b \in H^8(\Gr_2(\R^9))$ such that
$$
H^*(\Gr_2(\R^9)) = \frac{\Z[e, b]}{(e^4 -2b, b^2)}.
$$
See \cite[Section~9]{BottSpace} for a discussion of this result. 

The ring homomorphism $q^* : H^*(P) \to H^*(Q)$ makes $H^*(Q)$ into a module over $H^*(P)$. The following result gives the structure of $H^*(Q)$ as a module over $H^*(P)$.

\begin{lemma} \label{lem:cohofQ}
Let $i : \Gr_2(\R^9) \to Q$ be the inclusion of a fibre of $q : Q \to P$.
\begin{enumerate}
\item 
$i^* : H^*(Q) \to H^*(\Gr_2(\R^9))$ is surjective, and $q^* : H^*(P) \to H^*(Q)$ is injective. 
\item
Choose $s \in H^2(Q)$, $v \in H^8(Q)$ such that $i^*(s) = e$ and $i^*(v)= b$.  Then $H^*(Q)$ is a free module over $H^*(P)$ with basis
$$
1, \quad s, \quad s^2, \quad s^3, \quad v, \quad sv, \quad s^2v, \quad s^3v.
$$
\item 
The ring $H^*(Q)$ is generated by $s, v, a$.
\end{enumerate}
\end{lemma}

\begin{proof} 
Since both $H^*(P)$ and $H^*(\Gr_2(\R^9))$ are zero in odd degrees the Serre spectral sequence of the fibre bundle $q : Q \to P$ collapses at the $E_2$ page. The result follows from the Leray-Hirsch theorem.
\end{proof}

Since $q^* : H^*(P) = \Z[a]/(a^3) \to H^*(Q)$ is injective, from now on we regard $\Z[a]/(a^3)$ as a subring of $H^*(Q)$. We have a basis for $H^*(Q)$ and we need to calculate products between $s$, $v$, and $a$. To do this we use characteristic classes.  

\subsection{Products in $H^*(Q)$}
The real representations $\Spin(7) \times_{C_2} \Spin(2) \to \SO(7)$ and $\Spin(7) \times_{C_2} \Spin(2) \to \SO(2)$ define real homogeneous vector bundles $U_7$ and $U_2$ of dimensions $7$ and $2$ over $Q$.  The real representation $\Spin(9) \to \SO(9)$ defines a real vector bundle $U_9$ over $P$.  Evidently 
$$
U_7 \oplus U_2 = q^* (U_9)
$$
and applying characteristic classes to this identity of vector bundles will give us relations between products of $s, v, a$.

\begin{lemma} \label{lem:cohofQmore}
Let $i : \Gr_2(\R^9) \rightarrow Q$ be the inclusion of a fibre of $q: Q \to P$.  There exists a unique element $v \in H^8(Q)$ such that
$$
i^*(v) = b, \quad 2v = s^4 + a.
$$
\end{lemma}

\begin{proof}
Let $w$ be the total Stiefel-Whitney class.  Then applying $w$ to the above relation between bundles gives
$$ 
w(U_7) w(U_2) = q^* (w(U_9)).
$$  
Simple arguments show that in $H^*(P; \Z/2)$, and $H^*(Q; \Z/2)$ respectively,
$$
w(U_9) = 1 + a,  \quad w(U_2) = 1 + s \mod 2. 
$$
It follows that
$$
w(U_7)(1+s) = 1 + a \mod 2,
$$
and in particular
$$
w_8(U_7) = s^4 + a \mod 2.
$$
Since $U_7$ is $7$-dimensional, $w_8(U_7) = 0$ and we conclude that
$$
s^4 + a = 0 \mod 2.
$$
This shows that $s^4 + a \in H^{8}(Q)$ is divisible by $2$ and since $H^*(Q)$ is torsion free it is uniquely divisible by $2$.
\end{proof}

\begin{corollary}
Let $s$ and $v$ be as in the above lemma. The ring $H^*(Q)$ is generated by $s \in H^2(Q)$ and $v \in H^8(Q)$.
\end{corollary}
\begin{proof}
Lemma \ref{lem:cohofQ} shows that $H^*(Q)$ is generated by $s, a, v$ and Lemma \ref{lem:cohofQmore} shows that $2v = s^4 + a$. 
\end{proof}

This allows us to fix the generators of $H^*(Q)$, explicitly:
\begin{enumerate}
\item
$s \in H^2(Q)$ is the Euler class of $U_2$,
\item
$v \in H^8(Q)$ is the unique class such $i^*(v) = b$ and $2v  = s^4 + a$.
\end{enumerate}

We now need to find two relations between these generators, one in degree $16$ and the other in degree $24$.

\begin{lemma} In $H^{16}(Q)$ we have the relation
\[
s^8 = 3v^2.
\]
\end{lemma}

\begin{proof}
Let $p$ be the total Pontryagin class.  Since $H^*(Q)$ is torsion free, the relation $U_7 + U_2 = q^*(U_9)$ between bundles gives the relation
$$
p(U_2)p(U_7) = q^*(p(U_9)).
$$
Now $p(U_2) = 1 - s^2$ and from \cite[Theorem~19.4]{MR0102800} we know that $p(U_9) = 1 - 6a - 3a^2$.  It follows that
$$
p_4(U_7)  = s^8  - 6as^4 -3a^2.
$$
Since $U_7$ is $7$-dimensional, $p_i(U_7) = 0$ for $i \geq 4$ and so
$$
s^8  - 6as^4 -3a^2 =0.
$$
Substitute $a = 2v - s^4$ and this simplifies to
$$
4s^8 = 12v^2.
$$
Since $H^*(Q)$ is torsion free this proves the lemma.  
\end{proof}

The relation in degree $24$ is the obvious one coming from $P$: 
$$
a^3 = (2v - s^4)^3 = 0.
$$

\begin{theorem} \label{thm:intcohQ}
The integral cohomology ring of $Q$ is given by
\[
H^*(Q) = \frac{\Z[s,v]}{(\rho_{16},\rho_{24})},
\]
where $s \in H^2(Q)$, $v \in H^8(Q)$ are as above, and the relations are
\[
\rho_{16} = s^8 - 3v^2, \quad \rho_{24} = (2v - s^4)^3.
\]
\end{theorem}

\begin{proof}
We have constructed a surjective ring homomorphism
$$
\frac{\Z[s,v]}{(\rho_{16},\rho_{24})} \to H^*(Q).
$$
A simple counting argument shows the rank of the homogeneous degree $k$ part of $\Z[s,v] /(\rho_{16},\rho_{24})$ is free abelian with the same rank as $H^k(Q)$.  Therefore this ring homomorphism is an isomorphism.
\end{proof}

Finally we choose a generator $[Q] \in H^{30}(Q) = \Z$ and use this to write down the products in $H^*(Q)$ to the top dimension.
\begin{theorem} \label{thm:tabQ}
The products to $H^{30}(Q)$ are given by the following table.
\begin{center}
\begin{tabular}{ | c | c | c | c |}
\hline
$s^{15}$ & $s^{11}v$ & $s^7v^2$ & $s^3v^3 $\\
\hline
$3 \cdot 26$ & $3 \cdot 15$ & $26$ & $15$ \\
\hline
\end{tabular}
\end{center}
\end{theorem}

\begin{proof}
The abelian group $H^{30}(Q)$ has the following presentation.  There are four generators
\[
s^{15}, \quad s^{11}v, \quad s^7v^2, \quad s^3v^3
\]
and three relations 
\[s^7\rho_{16} = 0, \quad s^3v\rho_{16} = 0, \quad s^3\rho_{24} =0.\] 
We let
\[
a = s^{7}v^2, \quad b = s^3v^3. 
\]
Then the relations show that
$$
s^{15} = 3a, \quad s^{11}v = 3b, \quad 15a = 26b.
$$
So $H^{30}(Q)$ is the abelian group generated by $a, b$ with the single relation $15a = 26b$.  Since $15$ and $26$ are coprime this abelian group is isomorphic to $\Z$, and the isomorphism is the unique homomorphism such that $a \mapsto 26$ and $b \mapsto 15$.
This gives the values in the table. 
\end{proof}

\section{The integral cohomology ring of $R$.} \label{sec:HR}
We know that the cohomology of both $R$ and $Q$ is torsion free and zero in odd degrees. From \cite[Section~7.4]{JRTI}, we see that the Poincar\'e polynomial of $R$ is
\[
(1 + x^2 + x^4 + \dots + x^{16})(1 + x^8 + x^{16}),
\]
and recall from Section \ref{sec:intcohQ} that the Poincar\'e polynomial of $Q$ is
\[
(1 + x^2 + x^4 + \dots + x^{14})(1 + x^8 + x^{16}).
\]
The following table, which highlights the very small difference between $R$ and $Q$, lists the non-zero Betti numbers of $R$ and $Q$.

\vspace{\baselineskip}

\begin{center}
\begin{tabular}{| c | c | c | c | c | c | c | c | c | c | c | c | c | c | c | c | c | c |} 
\hline
2k & 0 & 2 & 4 & 6 & 8 & 10 & 12 & 14 & 16 & 18 & 20 & 22 & 24 & 26 & 28 & 30 & 32  \\
\hline
$b_{2k}(R_5)$ & 1 & 1 & 1 & 1 & 2 & 2 & 2 & 2 & 3 & 2 & 2 & 2 & 2 & 1 & 1 & 1 & 1 \\
\hline
$b_{2k}(Q) $ & 1 & 1 & 1 & 1 & 2 & 2 & 2 & 2 & 2 & 2 & 2 & 2 & 1 & 1 & 1 & 1 &  \\
\hline
\end{tabular}
\end{center}

\vspace{\baselineskip}

Let $i : P \to R$ be the embedding of $P$ in $R$. Then we have the usual induced homomorphisms $i_*,i^*$ in homology and cohomology, respectively.  However both $P$ and $R$ are orientable manifolds and $i$ is a codimension $16$ embedding, so we also have the umkehr homomorphisms 
$$
i_{!} : H^r(P) \to H^{r+16}(R).
$$
Recall that $i_{!}$ is defined by the following commutative diagram
$$
\begin{CD}
H^{r}(P) @>i_{!}>> H^{r+16}(R) \\
@VVV @VVV \\
H_{16-r}(Q) @>>{i_*}> H_{16-r}(R)
\end{CD}
$$
where the vertical arrows are the Poincar\'e duality isomorphisms. Now let $j : Q \to R$ be the codimension $2$ embedding of $Q$.  In this case we have $j_*, j^*$ and 
$$
j_! : H^r(Q) \to H^{r+2}(R).
$$ 
We know that the cohomology groups of $P, Q, R$ are zero in odd degrees.

\begin{lemma}
There are short exact sequences
$$
\begin{CD}
0 @>>> H^{2k-2}(Q) @>>{j_!}>  H^{2k}(R)@>>{i^*}> H^{2k}(P) @>>> 0,
\end{CD}
$$$$
\begin{CD}
0 @>>> H^{2k-16}(P) @>>{i_!}>  H^{2k}(R)@>>{j^*}> H^{2k}(Q) @>>> 0.
\end{CD}
$$
\end{lemma}
\begin{proof}
By Theorem \ref{mainthm:Mostert}, $R/P$ is the Thom space $\Th(N_Q)$ of the normal bundle of $j : Q \to R$.  The dimension of this bundle is $2$ so the Thom isomorphism
$$ 
H^{s-2}(Q) \to H^{s }(\Th(N_Q))
$$
shows $R/P = \Th(N_Q)$ has no cohomology in odd degrees.  Therefore the connecting homomorphism in the long exact sequence of the pair $(R,P)$ is always zero.  Finally the composite
$$
H^{s-2}(Q) \to  H^s(\Th(N_Q)) = H^s(R/P) \to H^{s}(R)
$$
is an alternative definition of $j_!$.  This gives the first exact sequence. The second follows by identifying $R/Q$ with $\Th(N_P)$ and repeating the argument in this context.
\end{proof}

\begin{corollary} \label{cor:jstarshreik}
\begin{enumerate}
\item The map $j^* : H^s(R) \to H^s(Q)$ is an isomorphism for $s \leq 15$.
\item The map $j_! : H^s(Q) \to H^{s+2}(R)$ is an isomorphism for $s \geq 17$.
\item There is a short exact sequence
$$
\begin{CD}
0 @>>> H^{14}(Q) @>>{j_!}>  H^{16}(R)@>>{i^*}> H^{16}(P) @>>> 0
\end{CD}
$$
\end{enumerate}
\end{corollary}

First, we define $t \in H^2(R)$ and $w \in H^8(R)$ by
$$
j^*(t) = s, \quad j^*(w) = v.
$$

\begin{lemma}
The ring $H^*(R)$ is generated by $t$ and $w$.
\end{lemma}
\begin{proof}
Suppose $x \in H^k(R)$ and $k \leq 15$. Then since $s, v$ generate $H^*(Q)$ the above corollary shows that $x = j^*(p(s,v))$ for some polynomial $p$ in two variables. By  definition, $j^*(s) = t$ and $j^*(v) = w$, so $x = p(t,w)$. Therefore $x$ is a polynomial in $t, w$. 

Now suppose $x \in H^k(R)$ and $k \geq 17$.  This time, the corollary shows that $x = j_!(q(s,v))$ for some polynomial $q$ in two variables.  Since $j^*(t) = s$ and $j^*(w) = v$ it follows that $q(s,v) = j^*(q(t,w))$.  Now the general properties of $j_!$ show that  
\[
 j^*(j_{!}(1)) = s, \quad  j_{!}(j^*(b)c) = bj_{!}(c), \quad j_{!}(j^*(b)) = bt.
\]
The first formula shows that $j_!(1) = t$ so the third formula follows from the first two in the special case $c=1$.  It follows that
\[
x = j_!(q(s, v)) = j_!(j^*(q(t, w))) =q(t, w)t.
\]
This shows that if $k \geq 17$ any $x \in H^k(R)$ can be written as a polynomial in $t$ and $w$.

We are left to prove that the same conclusion is true if $x \in H^{16}(R)$. We claim that $i^*(w^2) = a^2$. Assuming this claim is true, consider the exact sequence
\[
0 \to H^{14}(Q) \to H^{16}(R) \to H^{16}(P) \to 0
\]
in Corollary \ref{cor:jstarshreik}. Then $H^{14}(Q)$ is $\Z \oplus \Z$ with basis $s^7, s^3v$.  It follows that $j_!(s^7), j_!(s^3v), w^2$ is a basis for $H^{16}(R)$. Repeating the argument of the previous paragraph shows that $j_!(s^7) = t^8$ and $j_!(s^3v) = t^4w$.  So every element in $H^{16}(R)$ is also given by a polynomial in $t$ and $w$.

It remains to prove that $i^*(w^2) = a^2$.  Corollary \ref{cor:jstarshreik} shows that the homomorphism $i^*: H^8(R) \to H^8(P)$ is surjective.  We also know that $t^8, w$ is a basis for $H^8(R)$. For degree reasons $i^*(t) = 0$ and so $i^*(t^4) = 0$. Thus, $i^*(w) = \pm a$.
\end{proof}

Next we show how to compute the products to the top degree in $R$.  
\begin{theorem} \label{thm:prodsR}
The products to $H^{32}(R)$ are given by the following table.
\vspace{\baselineskip}
\begin{center}
\begin{tabular}{ | c | c | c | c | c | }
\hline
$t^{16}$ & $t^{12}w$ & $t^8w^2$ & $t^4w^3$ & $w^4$\\
\hline
$3 \cdot 26$ & $3 \cdot 15$ & $26$ & $15$ & $9$ \\
\hline
\end{tabular}
\end{center}
\end{theorem}

\begin{proof}
Note that 
\[
j_! : H^{30}(Q) \to H^{32}(R)
\]
is an isomorphism.  The first four entries in the table follow by applying $j_!$ to the entries in the table in Theorem \ref{thm:tabQ} and using the formula $j_!(j^*(x)) = xt$.  

Now we need to compute $x$, the entry in the table corresponding to $w^4$.  The abelian group $H^{16}(R)$ is free of rank $3$ with basis
\[
t^8, \quad t^4w, \quad w^2.
\]
The intersection matrix with respect to this basis is 
\[
\begin{pmatrix}
3 \cdot 26 & 3\cdot 15 &26 \\
3 \cdot 15 & 26            &15 \\
26             & 15            & x
\end{pmatrix}.
\]
This matrix must have determinant $\pm 1$. The determinant is $3x - 26$ so it follows that $x$ must be $9$.
\end{proof}

This completely determines the ring $H^*(R)$ but it would seem quite perverse not to extract the degree $18$ and degree $24$ relations it implies.  

\begin{lemma}
\begin{enumerate}
\item
In $H^{18}(R)$ we have the relation
\[
 t^9 - 3w^2t = 0.
\]
\item
In $H^{24}(R)$ we have the relation
\[
w^3 + 15w^2t^4 - 9wt^8 = 0.
\]
\end{enumerate}
\end{lemma}
\begin{proof}
To prove (i) we argue as follows.  In $H^{16}(Q)$
$$
j^*(t^8 - 3w^2) = s^8 - 3v^2 = 0
$$
and therefore
$$
0 =j_! ((j^*(t^8 - 3w^2)) = (t^8 - 3w^2)t.
$$

To prove (ii) we simply check that
$$
t^4(w^3 + 15w^2t^4 - 9wt^8) = w(w^3 + 15w^2t^4 - 9wt^8) = 0
$$
and it follows that $w^3 + 15w^2t^4 - 9wt^8 = 0$ by Poincar\'e duality. 
\end{proof}

We now complete the proof of Theorem \ref{mainthm:intcoh} in the same way we completed the proof of Theorem \ref{thm:intcohQ}. We have constructed a surjective ring homomorphism
\[
\frac{\Z[t,w]}{(r_{16}, r_{24})} \to H^*(R).
\]
We know that $H^k(R)$ is a free abelian group and we know its rank. A counting argument shows that the rank of the homogeneous degree $k$ part of $\Z[t, w]/(r_{16}, r_{24})$ is free abelian and has the same rank as $H^k(R)$.  

\section{The characteristic classes of the natural bundles over $R$} \label{sec:chern}

\subsection{Bundles over $R$}
We start with the relevant representation theory.  As usual we write $\rho_{10}$ for the $10$-dimensional vector representation of $\Spin(10)$, and $\delta_{10}^{\pm}$ for the two $16$-dimensional spin representations. As in the introduction, we use $(\eps, i)$ for the relevant central subgroup $C_4 \subset \Spin(10) \times S^1$. Recall that $\eps$ acts as multiplication by $\pm i$ on $\delta_{10}^\pm$ and as multiplication by $-1$ on $\rho_{10}$. A necessary and sufficient condition for a representation of $\Spin(10) \times S^1$ to descend to $\Spin(10) \times_{C_4} S^1$ is that this central subgroup $C_4$ acts trivially.  

Let  $\xi$ be the usual representation of $S^1$ on $\C$.  Then
$$
\rho_{10}\o \xi^2, \quad \delta_{10}^+\o \xi^3, \quad \delta_{10}^-\o \xi^{-3}, \quad \xi^4
$$
descend to $\Spin(10) \times_{C_4} S^1$ and so they define complex vector bundles 
$$
V, \quad D, \quad \bar{D}, \quad L
$$
of dimensions $10, 16, 16, 1$ over $R_5$, respectively. From \cite[Section~2]{MR0440584} we have the following relation between bundles over $R$
$$
L^{-1} + V + D \otimes L^{-1} = R \times \C^{27}.
$$
We know the Chern classes of $L$ so once we know the Chern classes of $D$ we know the Chern classes of $\bar{D}$ and $V$.
We choose to compute the Chern classes of $D$ since it is straightforward to check from \cite[Chapter~6]{MR1428422} that $T_c(R) = D$. We choose to give these Chern classes in terms of the presentation of Theorem \ref{mainthm:intcoh} rather than the presentation in \cite[Theorem~8.2.1]{JRTI} since this leads to the most compact formulas. To do this, we must compare the two presentations.

\subsection{Comparing the two presentations of the cohomology of $R$} \label{sec:prescompare}
We work in the ring 
$$
H^*(R; \Q) = \frac{\Q[t,w]}{(r_{18}, r_{24})}
$$ 
as in Theorem \ref{mainthm:intcoh}.

In \cite[Theorem~8.2.1]{JRTI} we give a presentation of $H^*(R; \Q)$ with two generators  $a_2 \in H^2(R;\Q)$, $a_8 \in H^8(R; \Q)$ and two relations.  It is straightforward to check that $t = 4a_2$ and when we substitute $a_2 = t/4$ the two relations become
\begin{align}
0 & = ta_8^2 +\frac{27}{4}t^5a_8 -\frac{39}{64}t^9, \\
%0 & = {}-512a_8^3 -23616t^4a_8^2 + 23976t^8a_8 - 1539t^{12}. 
0 & = a_8^3 + \frac{369}{8} t^4a_8^2 - \frac{2997}{64}t^8a_8 + \frac{1539}{512} t^{12}. 
\end{align}
We must calculate $a_8$ as a polynomial in $t, w$.

\begin{theorem} \label{thm:prescompare}
The element
$$
a_8 = 6w - \frac{27}{8} t^4
$$
is the unique indecomposable element of the ring $H^*(R; \Q)$ satisfying relations (1) and (2).
\end{theorem}
 
\begin{proof}
Let $x \in H^8(R; \Q)$ be such that
$$
x = \lambda t^4 + \mu w, \quad tx^2 + \alpha t^5x = \beta t^9 
$$
where $\alpha, \beta, \lambda, \mu \in \Q$ and $\mu \neq 0$.  The condition $\mu \neq 0$ ensures that $x$ is indecomposable.

We work with the basis $u = t^9$, $v = t^5w$ for $H^{18}(R; \Q)$. Calculating $tx^2 + \alpha t^5x = \beta t^9$ in this basis (using $r_{18}$), and comparing the coefficients of $u$ and $v$ gives two equations:  
\begin{align*}
2 \lambda \mu + \alpha \mu & = 0,\\
\lambda^2 + \frac{\mu^2}{3} +\alpha \lambda & = \beta.
\end{align*}
Since $\mu \neq 0$, we find that 
$$
\lambda = -\frac{\alpha}{2} \quad \text{ and } \quad \mu^2 = 3 \left( \beta +\frac{\alpha^2}{4}\right).
$$
Fixing $\alpha = 27/4$ and $\beta = 39/64$ leads to $\mu = \pm 6$ and $\lambda = -27/8$. Therefore
\[
a_8 = \pm6w - \frac{27}{8}t^4. 
\]
A final calculation shows that 
$$
a_8 = 6w - \frac{27}{8}t^4
$$ 
is the only solution to relation (2). 
\end{proof}

\subsection{Calculating characteristic classes}
In \cite[Section~8.2]{JRTI} we construct a homomorphism
$$
\phi : H^*(B\Spin(10) \times S^1; \Q) \to H^*(R; \Q)
$$
such that
$$
\phi(c_k( \delta_{10}^+\otimes \xi^3)) = c_k(T_c(R)).
$$ 
In Section~8.1 of \cite{JRTI} we explain how to calculate $\phi$ in terms of the standard choice of generators for the ring $H^*(B\Spin(10) \times S^1; \Q)$ and the generators $a_2$ and $a_8$ of $H^*(R; \Q)$ referred to in the previous section.  We complete the proof of Theorem \ref{mainthm:chernclasses} in three steps.
\begin{enumerate}
\item 
First compute $c_k(\delta_{10}^+\otimes \xi^3)$ using the splitting principle, see \cite[Section~8.6]{JRTI}.
\item
Next compute $\phi(c_k(\delta_{10}^+\otimes \xi^3))$ in terms of the generators $a_2, a_8$.
\item
Finally use Section \ref{sec:prescompare} to express $c_k(T_c(R))$ in terms of the generators $t, w$.
\end{enumerate}

\printbibliography
\end{document}